\documentclass[11pt]{article}

\usepackage[a4paper,margin=2.5cm]{geometry}
\usepackage[T1]{fontenc}
\usepackage[utf8]{inputenc}
\usepackage{lmodern}
\usepackage{amsmath,amssymb,amsthm}
\usepackage{graphicx}
\usepackage{booktabs}
\usepackage{tikz}
\usetikzlibrary{positioning}
\usepackage[sort&compress,numbers]{natbib}
\usepackage[hidelinks]{hyperref}
\usepackage{microtype}

\numberwithin{equation}{section}
\newtheorem{definition}{Definition}[section]
\newtheorem{proposition}[definition]{Proposition}
\newtheorem{theorem}[definition]{Theorem}
\newtheorem{lemma}[definition]{Lemma}

\newtheorem{remark}[definition]{Remark}

\title{A Structural Theory of Admissible Transitions in Biological Reaction Networks}
\author{Stephan Peter$^{1}$ and Bashar Ibrahim$^{2,3,*}$\\[0.5em]
\small $^{1}$Department of Basic Sciences, Ernst-Abbe University of Applied Sciences Jena,\\
\small 07745 Jena, Germany\\
\small $^{2}$Department of Mathematics and Computer Science, Friedrich Schiller University Jena,\\
\small 07743 Jena, Germany\\
\small $^{3}$European Virus Bioinformatics Center, 07743 Jena, Germany\\
\small $^{*}$Corresponding author: stephan.peter@eah-jena.de}
\date{}

\begin{document}
\maketitle

\begin{abstract}
Biological reaction networks often exhibit complex transient behavior that cannot be explained solely by the analysis of steady states or long-term persistence. Existing structural approaches identify persistent system properties and have considered transitions between organizations, but do not provide a general
criterion for admissible transitions between arbitrary species configurations. We introduce admissible transitions, a new structural concept that describes feasible changes between species subsets using only reaction network structure and feasible reaction fluxes, independently of kinetic parameters. We prove that solutions of reaction-based ordinary differential equation systems induce canonical sequences of admissible transitions with a fundamental asymmetry: closure-building transitions are uniquely determined by network structure, whereas downward transitions are generally non-unique and depend on the realized trajectory. This establishes a structural layer linking network topology to transient system evolution. The framework is illustrated using a classical HIV immune-response model. By extending structural reaction network analysis from persistent states to transient dynamics, the proposed theory provides a general, kinetics-independent framework for analyzing reachability, organization formation, and transient behavior in biological systems.
\end{abstract}


\section{Introduction}
\label{sec:introduction}
Complex systems governed by interacting components often exhibit rich 
nonlinear dynamics with nontrivial transient behavior. While significant 
progress has been made in understanding asymptotic states and stability, 
the structural pathways through which systems evolve between different 
configurations remain less well understood. In particular, characterizing 
which transitions between system states are intrinsically permitted by 
network structure, independently of specific kinetic realizations, is a 
fundamental challenge in the study of dynamical systems. Reaction networks 
constitute a central modeling framework for such systems in chemistry, systems 
biology, ecology, and related fields. They support deterministic ordinary 
differential equations (ODEs), stochastic Markov processes, and other 
dynamical descriptions of interacting species. Classical approaches to 
reaction network analysis have mainly focused on asymptotic or steady--state 
behavior. Chemical reaction network theory (CRNT) relates network topology 
to qualitative properties of equilibria under specific kinetic assumptions 
\cite{Feinberg1979lectures,Feinberg1987deficiency,feinberg1995existence}, 
while chemical organization theory (COT) characterizes structurally stable 
sets of species independently of kinetic parameters 
\cite{dittrich2007chemical,peter2021linking,peter2023persistent}. 
Organizations were further shown to provide structural descriptors of persistent
species sets under broad dynamical assumptions \cite{peter2021linking,ibrahim2023persistent}. However, these approaches primarily address asymptotic behavior and transitions between asymptotic states resp. organizations, rather than transient structural transitions between arbitrary species configurations. In particular, they do not characterize which transitions between species
subsets are structurally admissible independently of kinetic laws or detailed
dynamical realizations. Related approaches such as elementary flux mode analysis
\cite{schuster1994elementary,klamt2003two} and Petri net methods
\cite{reisig2013understanding,murata1989petri} focus on steady--state fluxes,
reachability, or concentration-based state spaces rather than on abstract transitions
between species presence sets. Likewise, transient dynamics have been studied in
the context of bifurcation theory and invariant sets
\cite{hirsch2013differential,smith1995monotone}, but not through a general
structural transition framework.

In this work, we introduce a theory of \emph{admissible transitions} between subsets
of species. Transitions are defined independently of concentrations and kinetic laws
and are characterized structurally through stoichiometry and feasible reaction fluxes.
This separates structural constraints from the particular dynamical realization.

We show that solutions of reaction--based ODE systems induce canonical sequences
of feasibly admissible transitions exhibiting a characteristic asymmetry: upward
closure--building transitions are uniquely determined by network structure, whereas downward transitions associated with loss of reaction support are
generally non--unique and path--dependent. In this way, organizations emerge as
terminal nodes within a broader transition landscape rather than as isolated
descriptors of persistence.

The framework is illustrated using an HIV--immune system model, where identical
closure--building behavior can lead either to viral control or immune collapse
depending on the realized admissible downward transition. This demonstrates how
transient structural pathways may influence long--term biological outcomes.


The main contributions of this work are:
\begin{itemize}
\item a structural definition of admissible and feasibly admissible transitions between species subsets;
\item the identification of canonical transition sequences induced by solutions of reaction--based ODE systems;
\item the proof of a fundamental asymmetry between structurally determined upward transitions and path--dependent downward transitions;
\item the integration of organizations into a transition framework extending structural analysis beyond asymptotic behavior.
\end{itemize}

This perspective provides a bridge between continuous nonlinear dynamics and discrete structural transitions in network-driven systems.


The remainder of the paper is organized as follows.
Section~2 introduces the reaction-network background, notation, feasible
reaction fluxes, and the relation between organizations and persistence.
Section~3 develops the theory of admissible transitions, establishes their
basic structural properties, and connects the discrete transition
framework to reaction-based ODE systems.
Section~4 illustrates the framework using the Wodarz--Nowak HIV--immune
system model and discusses its relation to stochastic Markov descriptions.
Section~5 places the proposed approach in the context of related structural
methods and discusses its biological, computational, and mathematical
implications, limitations, and possible extensions.
Finally, Section~6 summarizes the main conclusions and outlines the role
of admissible transitions as a structural description of transient
dynamics in biological reaction networks.


\section{Methods}
\label{sec:preliminaries}

This section recalls only those structural notions from reaction network theory
that are required for the development of the transition framework in Section~\ref{sec:results}.
All definitions are purely structural and independent of any particular choice
of kinetics or dynamical realization.
Background on reaction network theory and chemical organization theory can be found,
for instance, in
\cite{Feinberg1979lectures,Feinberg1987deficiency,dittrich2007chemical,peter2021linking}.

\subsection{Reaction networks and stoichiometry}
\label{subsec:reaction-networks}

A reaction network consists of a finite set of species
\[
S=\{s_1,\dots,s_n\}
\]
and a finite set of reactions
\[
R=\{r_1,\dots,r_m\}.
\]
Each reaction $r_j$ is represented by its stoichiometric input and output vectors,
and the network is encoded by its stoichiometric matrix
\[
N\in\mathbb{Z}^{n\times m}.
\]
The entry $N_{ij}$ denotes the net change of species $s_i$ caused by reaction $r_j$.
This representation is standard in chemical reaction network theory
\cite{Feinberg1979lectures,Feinberg1987deficiency}.

Throughout this work, the term reaction--based dynamical system refers to any
dynamical system whose evolution is induced by a fixed reaction network via
reaction fluxes.
No assumptions on kinetic laws are made unless explicitly stated.

\subsection{Presence sets and abstraction}
\label{subsec:presence-sets}

For a concentration vector $x\in\mathbb{R}_{\ge 0}^n$, we define its presence set by
\begin{equation}
\label{eq:presence-set}
\varphi(x)\;:=\;\{\, s_i\in S \mid x_i>0 \,\}.
\end{equation}

This abstraction associates to each continuous state the set of species that are
present at positive concentration.
It is introduced solely to relate continuous concentration dynamics to set--valued
presence patterns in reaction--based ODE systems.
It plays no role in the purely structural definitions of transitions introduced later.

\subsection{Feasible reaction fluxes}
\label{subsec:feasible-fluxes}

Reaction network structure constrains which reaction fluxes are compatible with a
given presence set.

\begin{definition}[Feasible flux]
\label{def:feasible-flux}
Let $A\subseteq S$.
A vector $\hat v\in\mathbb{R}^m_{\ge 0}$ is called a feasible flux with respect to $A$
if every reaction $r_j$ with $\hat v_j>0$ has all its reactant species contained in $A$.
\end{definition}

Feasibility is a purely structural notion and does not imply that the corresponding
flux is dynamically realized by a particular kinetic law.
It captures which reactions are structurally enabled by a given presence set and is
closely related to support-based notions used in chemical organization theory
\cite{dittrich2007chemical,peter2021linking}.

\subsection{Organizations and self--maintenance}
\label{subsec:organizations}

Organizations provide structurally distinguished reference sets for persistent behavior
in reaction networks.

\begin{definition}[Organization]
A subset $O\subseteq S$ is called an organization if it is
\begin{enumerate}
\item closed, i.e.\ no reaction produces a species outside $O$ when restricted to $O$, and
\item self--maintaining, i.e.\ there exists a nonnegative flux vector $v\ge 0$ such that
$(Nv)_i\ge 0$ for all $s_i\in O$.
\end{enumerate}
\end{definition}

Organizations were introduced as structural descriptors of persistent behavior in
reaction networks in \cite{dittrich2007chemical} and further developed in
\cite{peter2021linking}.
They are invariant under the reaction structure and serve as canonical structural
reference points for terminal behavior.
In the following, organizations are used exclusively as reference points for zero
transitions and terminal configurations.

\subsection{Organizations and persistence}
\label{subsec:persistence}

Chemical organization theory (COT) characterizes structurally stable subsets
of species in reaction networks through the notions of closure and
self--maintenance \cite{dittrich2007chemical}. Organizations were later
shown to be closely related to persistence properties of reaction--based
dynamical systems, including ordinary differential equations and
reaction--diffusion systems \cite{peter2021linking}.

We briefly recall two results from \cite{peter2021linking} that
provide the dynamical foundation for the transition framework developed
later in this work.

First, persistent species sets in the sense of Definition~3.5~\cite{peter2021linking} are necessarily organizations of the
underlying reaction network.

\begin{theorem}[Persistent species sets are organizations {\cite[Theorem~3.25]{peter2021linking}}]
Let $(S,R)$ be a reaction network and let $c$ be a solution of a dynamical system with underlying reaction network $(S,R)$.
Then the set
\[
P(c)=\{s\in S \mid s \text{ is persistent with respect to } c\}
\]
is an organization of $(S,R)$.
\label{th1}
\end{theorem}

Thus, organizations provide canonical structural descriptors of asymptotic
persistent behavior.

Second, closure building occurs immediately once reactions producing new
species become structurally enabled.

\begin{lemma}[Instant appearance of closure {\cite[Lemma~3.21]{peter2021linking}}]
Let $c$ be a solution of a dynamical system with underlying
reaction network $(S,R)$. Then every species contained in the closure
$\mathrm{clos}(\phi(c(t)))$ of the current presence set appears immediately.
More precisely, if
\[
s_i \in \mathrm{clos}(\phi(c(t))),
\]
then for every $\varepsilon>0$ there exists a time
$t' \in (t,t+\varepsilon)$ such that
\[
c_i(t')>0.
\]
\label{l1}
\end{lemma}

This result implies that closure-building transitions are structurally
canonical and dynamically unavoidable once the corresponding reactions are
enabled.

Together, these results establish a characteristic asymmetry:
upward closure-building behavior is canonically induced by network structure,
whereas the eventual terminal organization may depend on the realized
dynamical trajectory. However, the cited results do not describe the
intermediate structural transitions between species subsets. This motivates
the transition framework introduced in the next section, which captures
structurally admissible changes between species subsets independently of
specific kinetic realizations.

\section{Results}
\label{sec:results}

\subsection{Structural transitions between species sets}
\label{subsec:transitions-sets}

To formalize structural changes in network-driven dynamical systems, we introduce
transitions as purely set-theoretic objects on the power set of species.
No kinetic assumptions, time parametrization, or dynamical realizability are imposed;
these aspects enter only later when we restrict to structurally admissible transitions
induced by a reaction network.

\begin{definition}[Directed transition]
\label{def:directed-transition}
Let $S$ be a finite set and let $A,B\subseteq S$. A (directed) transition from $A$ to $B$ is the ordered pair
$T=(A,B)$, written $A\to B$. We call $A$ the source set and $B$ the target set of $T$.
Two transitions $(A,B)$ and $(A',B')$ are equal if $A=A'$ and $B=B'$.
\end{definition}

\begin{definition}[Binary vector representation]
\label{def:binary-representation}
Fix an ordering $S=\{s_1,\dots,s_n\}$. For any subset $A\subseteq S$ we define its indicator vector
$\vec A\in\{0,1\}^n$ by
\[
(\vec A)_i :=
\begin{cases}
1, & s_i\in A,\\
0, & s_i\notin A,
\end{cases}
\qquad i=1,\dots,n .
\]
\end{definition}

\begin{definition}[Partial composition of transitions]
\label{def:composition}
Let $A,B,C\subseteq S$. For transitions $T_{AB}=(A,B)$ and $T_{BC}=(B,C)$ we define their composition by
\[
T_{AB}\circ T_{BC} := (A,C).
\]
Composition is defined only when the target of the first transition equals the source of the second.
Accordingly, any finite sequence $A_0\to A_1\to\cdots\to A_k$ is understood to satisfy $A_{i+1}$ being the
target of the $i$-th transition and the source of the $(i+1)$-st transition.
\end{definition}

\begin{remark}[Set-theoretic transitions vs.\ network-induced transitions]
\label{rem:set-theoretic-vs-admissible}
Definitions~\ref{def:directed-transition}--\ref{def:composition} describe transitions purely on $\mathcal P(S)$.
In later subsections we will restrict attention to those transitions that are induced by a reaction network
in the sense of structural admissibility. The partial composition above is compatible with concatenating
such admissible transitions, but admissibility itself is not automatic and is not symmetric in general.
\end{remark}

\subsection{Difference vectors and transition classes}
\label{subsec:difference-vectors}

Transitions between species sets encode discrete structural changes in the
state space of the underlying dynamical system. To represent these changes
algebraically, we associate each transition with a ternary difference vector
derived from the indicator representation introduced in Definition~\ref{def:binary-representation}.

\begin{definition}[Difference vector]
\label{def:difference-vector}
Let $S=\{s_1,\dots,s_n\}$ be a finite set and let $A,B\subseteq S$.
The difference vector associated with the transition $T=(A,B)$ is defined as
\[
\Delta(T):=\vec B-\vec A \in \{-1,0,+1\}^n,
\]
where $\vec A,\vec B\in\{0,1\}^n$ denote the indicator vectors of $A$ and $B$,
respectively.
\end{definition}

Each component $\Delta(T)_i$ specifies whether species $s_i$ disappears ($-1$),
appears ($+1$), or remains unchanged ($0$) along the transition.

\begin{definition}[Support sets of a transition]
\label{def:support-sets}
Let $T=(A,B)$ be a transition with difference vector $\Delta(T)$.
We define
\[
P(T):=\{\,i\mid \Delta(T)_i=+1\,\},
\qquad
N(T):=\{\,i\mid \Delta(T)_i=-1\,\}.
\]
Equivalently,
\[
P(T)=B\setminus A,
\qquad
N(T)=A\setminus B.
\]
\end{definition}

The sets $P(T)$ and $N(T)$ describe which species are structurally added or removed
by the transition.

\begin{definition}[Transition classes]
\label{def:transition-classes}
Let $T=(A,B)$ be a transition.
\begin{enumerate}
\item
$T$ is called a \emph{zero transition} if $\Delta(T)=0$, equivalently $A=B$.
Otherwise, $T$ is called \emph{non-zero}.

\item
$T$ is called \emph{non-positive} if
\[
\Delta(T)\in\{-1,0\}^n,
\]
equivalently if $P(T)=\emptyset$ or $B\subseteq A$.

\item
$T$ is called \emph{non-negative} if
\[
\Delta(T)\in\{0,+1\}^n,
\]
equivalently if $N(T)=\emptyset$ or $A\subseteq B$.

\item
$T$ is called \emph{plus-containing} if $P(T)\neq\emptyset$ and
\emph{minus-containing} if $N(T)\neq\emptyset$.
\end{enumerate}
\end{definition}

All notions introduced in this subsection depend only on set-theoretic presence
and absence of species. No reaction network, stoichiometric, or dynamical assumptions
are involved at this stage.

The transition classes defined above are purely combinatorial and apply to arbitrary
finite sets. We now specialize to reaction networks by introducing structurally
admissible transitions compatible with stoichiometric constraints and reaction support.

\subsection{Structurally admissible transitions in reaction networks}
\label{subsec:admissible-transitions}

For reaction networks, transitions between subsets of species are constrained 
by the underlying interaction structure. From the perspective of dynamical systems, 
these constraints determine which discrete changes in the system’s state space are 
structurally permissible. In this subsection, we derive criteria that characterize 
such admissible transitions based solely on stoichiometry and reaction support.

A key structural quantity is the stoichiometric effect vector $N v$, where $N$ denotes the
stoichiometric matrix of the reaction network and $v\in\mathbb{R}^m_{\ge 0}$ is a nonnegative
reaction flux vector. Since reaction support matters for realizability from a given source set,
we distinguish two layers of admissibility:

\begin{itemize}
\item Stoichiometric admissibility, which only requires the existence of some
nonnegative flux vector $v\ge 0$ that matches the sign pattern of a transition.
This layer is purely algebraic and does not encode which reactions are enabled at the source set.

\item Feasible admissibility, which in addition requires that the witnessing flux vector
$\hat v$ is a feasible flux with respect to the source set in the sense of
Definition~\ref{def:feasible-flux}. This layer incorporates reaction support (enabled reactions) and is therefore
inherently source-dependent.
\end{itemize}

In the subsequent subsubsections, we derive relations between admissible transitions and
structural properties of species subsets such as closedness, self-maintenance and organizations.

\subsubsection{Structural admissibility}
\label{subsubsec:structural-admissibility}

We fix a reaction network $(S,R)$ with stoichiometric matrix
$N\in\mathbb Z^{n\times m}$.
Let $T=(A,B)$ be a transition between subsets $A,B\subseteq S$ and let
\[
\delta(T)=\vec B-\vec A\in\{-1,0,+1\}^n
\]
denote its difference vector
(Definition~\ref{def:difference-vector}).

Admissibility is tested only for species not contained in both sets.
We therefore define
\begin{equation}
\label{eq:test-index-set}
J(T):=\{\,i\in\{1,\dots,n\}\mid s_i\notin A\cap B\,\}.
\end{equation}

We use the sign map
\[
\operatorname{sign}:\mathbb R\to\{-1,0,+1\},
\qquad
\operatorname{sign}(0)=0.
\]

\begin{definition}[Stoichiometrically elementary admissible transition]
\label{def:stoich-elementary-adm}
A transition $T=(A,B)$ is called stoichiometrically elementary admissible
if there exists a vector $v\in\mathbb R_{\ge0}^m$ such that
\begin{equation}
\label{eq:stoich-elementary}
\operatorname{sign}\big((Nv)_i\big)=\delta(T)_i
\qquad
\text{for all } i\in J(T).
\end{equation}
\end{definition}

\begin{definition}[Feasibly elementary admissible transition]
\label{def:feasible-elementary-adm}
A transition $T=(A,B)$ is called feasibly elementary admissible
if there exists a feasible flux
$\hat v\in\mathbb R_{\ge0}^m$
with respect to the source set $A$
(Definition~\ref{def:feasible-flux}) such that
\begin{equation}
\label{eq:feasible-elementary}
\operatorname{sign}\big((N\hat v)_i\big)=\delta(T)_i
\qquad
\text{for all } i\in J(T).
\end{equation}
\end{definition}

Admissibility of general transitions is defined via finite compositions of
elementary admissible transitions.

\begin{definition}[Stoichiometrically admissible transition]
\label{def:stoich-adm}
A transition $T=(A,B)$ is called stoichiometrically admissible if there exist
$k\ge1$ and subsets
\[
A=A_0,A_1,\dots,A_k=B
\]
such that each transition $(A_{j-1},A_j)$ is stoichiometrically elementary admissible.
\end{definition}

\begin{definition}[Feasibly admissible transition]
\label{def:feasible-adm}
A transition $T=(A,B)$ is called feasibly admissible if there exist
$k\ge1$ and subsets
\[
A=A_0,A_1,\dots,A_k=B
\]
such that each transition $(A_{j-1},A_j)$ is feasibly elementary admissible.
\end{definition}

\begin{proposition}[Closure under composition]
\label{prop:composition-admissibility}
Let $A,B,C\subseteq S$.
If the transitions $A\to B$ and $B\to C$ are feasibly admissible, then the
composed transition $A\to C$ is feasibly admissible.
The analogous statement holds for stoichiometrically admissible transitions.
\end{proposition}

\begin{proof}
By Definition~\ref{def:feasible-adm}, the transition $A\to B$ admits a finite
decomposition into feasibly elementary admissible transitions, and so does
$B\to C$. Concatenating these two decompositions gives a finite decomposition
of $A\to C$ into feasibly elementary admissible transitions. Hence $A\to C$
is feasibly admissible. The stoichiometric case is identical.
\end{proof}

Every feasibly admissible transition is stoichiometrically admissible,
whereas the converse need not hold since feasibility depends on reaction
support at the source set. Consequently, feasible admissibility is
generally source-dependent and need not be symmetric under reversal of
transitions.

In the following, the term admissible is understood in the feasible sense
unless explicitly stated otherwise.

We also say that a source set $A$ causes an admissible transition
$T=(A,B)$ whenever $T$ is admissible with source set $A$.

\subsubsection{Relations between admissible transitions and species subsets}
\label{subsubsec:relations-adm-subsets}

We now relate admissible transitions to structural properties such as closedness,
self--maintenance, and organizations. Throughout this subsection, admissibility is
understood in the feasible sense.

\begin{proposition}
\label{cor:organization-zero}
Let $A\subseteq S$.
\begin{enumerate}
\item
If $A$ is an organization, then there exists a feasibly admissible zero transition
$A\to A$.

\item
If $A$ is non-closed, then there exists a feasibly admissible plus-containing
transition $A\to B$.

\item
If $A$ is closed but not self-maintaining, then there exists a feasibly admissible
non-positive transition $A\to B$ with $B\subsetneq A$.
\end{enumerate}
\end{proposition}

Thus, non-organizations are structurally characterized by the existence of admissible
non-zero transitions, either through admissible expansion or admissible reduction.

The closure construction induces canonical admissible upward transitions.

\begin{proposition}
\label{cor:canonical-closure}
Let $A\subseteq S$ and let $\mathrm{clos}(A)$ denote its closure.
Then there exists a feasibly admissible plus-containing transition
\[
A \to \mathrm{clos}(A).
\]
More generally, all intermediate closure steps
\[
A \to \mathrm{clos}_1(A) \to \cdots \to \mathrm{clos}(A)
\]
are feasibly admissible.
\end{proposition}

These results are purely structural and independent of specific dynamical systems.
The asymmetry induced by feasibility is essential: while stoichiometric admissibility
depends only on sign patterns, feasible admissibility additionally depends on which
reactions are enabled at the source set.

In the next subsection, we show that solutions of reaction--based differential
equation systems induce canonical sequences of feasibly admissible transitions,
with the upward closure-building phase uniquely determined by network structure.

\subsection{Transitions induced by reaction--based ODE systems}
\label{subsec:ode-induced-transitions}

We now establish the connection between continuous nonlinear dynamics and the
discrete structural transitions introduced above. Specifically, we relate
admissible transitions to solutions of reaction-based ordinary differential
equation (ODE) systems. This connection shows how trajectories in the continuous
state space induce canonical sequences of discrete structural changes.

We consider initial value problems of the form
\begin{equation}
\label{eq:ode}
\dot c(t)=N\,v(c(t)),
\qquad
c(0)=c_0\in\mathbb R_{\ge0}^n,
\end{equation}
where $N\in\mathbb R^{n\times m}$ is the stoichiometric matrix and
$v:\mathbb R_{\ge0}^n\to\mathbb R_{\ge0}^m$ is a locally Lipschitz reaction
rate function.

We furthermore assume the standard reactant-availability condition that,
for every reaction $r_j$,
\[
v_j(c)>0
\quad\Longrightarrow\quad
\operatorname{React}(r_j)\subseteq \phi(c).
\]

To each concentration vector we associate its presence set
\[
\phi(c):=\{\,s_i\in S\mid c_i>0\,\}.
\]

Changes of $\phi(c(t))$ induce a canonical transition sequence
\[
\phi(c(t_0))
\to
\phi(c(t_1))
\to
\phi(c(t_2))
\to \cdots .
\]

Let
\[
S_0 := \phi(c(0))
\]
denote the initial presence set. By Lemma~2.4, species contained in
$\operatorname{clos}(S_0)$ appear once the corresponding reactions become
enabled, so that the closure-building phase is structurally determined by
the reaction network. Moreover, Theorem~2.3 implies that the persistent
species set $P(c)$ associated with a trajectory is an organization of the
underlying reaction network. We denote this organization by
\[
O := P(c).
\]
It is important to distinguish this asymptotic persistence statement from
finite-time realization of $O$ as a presence set. Accordingly, the
closure-building transition is obtained directly from the network
structure, whereas a transition from $\operatorname{clos}(S_0)$ to a
terminal organization $O$ is considered below whenever both sets are
realized as presence sets along the trajectory.

This canonical decomposition is illustrated schematically in
Figure~\ref{fig:phi_comuting_diagram}.

\begin{figure}[h]
\centering
\begin{tikzpicture}[
  >=latex,
  setnode/.style={draw, rounded corners, align=center, font=\small, inner sep=4pt},
  lab/.style={font=\footnotesize, inner sep=1pt, fill=white},
  arrow/.style={->, thick}
]

\node[setnode] (S0) at (0,0) {$S_0$};
\node[setnode] (C0) at (5,1.2) {$\mathrm{clos}(S_0)$};
\node[setnode] (O) at (10,0) {$O$};

\node[setnode] (c0)   at (0,-2.0) {$c(t_0)$};
\node[setnode] (c1)   at (5,-2.0) {$c(t_1)$};
\node[setnode] (cinf) at (10,-2.0) {$c(t_2)$};

\draw[arrow] (c0)   -- node[lab, pos=0.5, right] {$\phi$} (S0);
\draw[arrow] (c1)   -- node[lab, pos=0.5, right] {$\phi$} (C0);
\draw[arrow] (cinf) -- node[lab, pos=0.5, right] {$\phi$} (O);

\draw[arrow] (c0)   -- node[lab, pos=0.5] {$+N\cdot \int_{t_0}^{t_1} v(c(t))\,dt$} (c1);
\draw[arrow] (c1)   -- node[lab, pos=0.5] {$+N\cdot\int_{t_1}^{t_\infty}v(c(t))\,dt$} (cinf);

\draw[arrow, style=dashed] (S0) -- node[lab, pos=0.5] {$+\Delta(S_0,clos(S_0))$} (C0);
\draw[arrow] (C0) -- node[lab, pos=0.5] {$+\Delta(clos(S_0),O)$} (O);

\node[lab] (T0) at (0,2) {$\mathbf{t_0=0}$};
\node[lab] (T1) at (5,2) {$\mathbf{0<t_1<\infty}$};
\node[lab] (Tinf) at (10,2) {$\mathbf{t_2\to\infty}$};

\end{tikzpicture}
\caption{Transition structure associated with reaction-based ODE systems.
The closure-building transition from $S_0$ to $\operatorname{clos}(S_0)$
is structurally determined. If a terminal organization $O$ is realized as
a presence set along the trajectory, the subsequent transition from
$\operatorname{clos}(S_0)$ to $O$ is also feasibly admissible.}
\label{fig:phi_comuting_diagram}
\end{figure}
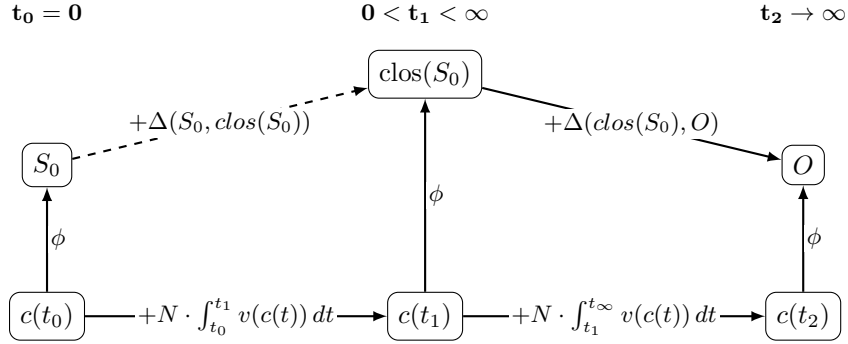

\begin{theorem}
\label{thm:ode-induced-transitions}
Let $c(t)$ be a solution of~\eqref{eq:ode} with initial presence set
\[
S_0=\phi(c(0)).
\]
Then the closure-building transition
\[
S_0 \longrightarrow \operatorname{clos}(S_0)
\]
is feasibly admissible.

Suppose, in addition, that there exist times $0\leq t_1<t_2$ such that
\[
\phi(c(t_1))=\operatorname{clos}(S_0)
\qquad\text{and}\qquad
\phi(c(t_2))=O,
\]
where $O$ is an organization. Then the transition
\[
\operatorname{clos}(S_0)\longrightarrow O
\]
is also feasibly admissible.
\end{theorem}

\begin{proof}
By Proposition~3.14, the closure construction can be decomposed into
feasibly elementary admissible transitions. Together with
Proposition~3.12, this implies that
\[
S_0 \longrightarrow \operatorname{clos}(S_0)
\]
is feasibly admissible.

For the second statement, assume that there exist
$t_1<t_2$ such that
\[
\phi(c(t_1))=\operatorname{clos}(S_0)
\qquad\text{and}\qquad
\phi(c(t_2))=O.
\]
Define the integrated reaction flux
\[
\widehat v :=
\int_{t_1}^{t_2} v(c(t))\,dt
\in\mathbb{R}_{\geq0}^{m}.
\]
Integration of~\eqref{eq:ode} over $[t_1,t_2]$ yields
\[
c(t_2)-c(t_1)=N\widehat v.
\]

Since $\operatorname{clos}(S_0)$ is closed, no species outside
$\operatorname{clos}(S_0)$ can be produced along this part of the
trajectory. Hence every reaction contributing positively to
$\widehat v$ has all its reactants contained in
$\operatorname{clos}(S_0)$. Therefore $\widehat v$ is a feasible flux
with respect to the source set $\operatorname{clos}(S_0)$.

For every species belonging to
$\operatorname{clos}(S_0)\setminus O$, we have
$c_i(t_1)>0$ and $c_i(t_2)=0$, and therefore
\[
(N\widehat v)_i
=
c_i(t_2)-c_i(t_1)<0.
\]
For every index
\[
i\in
J\bigl(\operatorname{clos}(S_0),O\bigr),
\]
the sign of $(N\widehat v)_i$ therefore agrees with the corresponding
component of
$\Delta(\operatorname{clos}(S_0),O)$.
No sign condition is imposed on species belonging to
$\operatorname{clos}(S_0)\cap O$, according to
Definition~3.9. Consequently,
\[
\operatorname{sign}\bigl((N\widehat v)_i\bigr)
=
\Delta\bigl(\operatorname{clos}(S_0),O\bigr)_i
\]
for all indices relevant to admissibility. Hence
\[
\operatorname{clos}(S_0)\longrightarrow O
\]
is feasibly elementary admissible, and therefore feasibly admissible.
\end{proof}
\begin{remark}
The additional finite-time realization assumption in
Theorem~\ref{thm:ode-induced-transitions} is distinct from asymptotic
persistence. Theorem~2.3 guarantees that the persistent species set of a
trajectory is an organization, but it does not imply that species outside
this organization reach zero concentration at a finite time. In numerical
or experimental applications, this distinction can naturally be handled
by a thresholded presence map such as the one used in Section~4.1.
\end{remark}

Thus, solutions of reaction--based ODE systems induce canonical sequences of
feasibly admissible transitions. In particular, the upward closure-building phase
is structurally determined, whereas the terminal phase may depend on the realized trajectory.


\section{Examples}

\subsection{Example for ODE systems: transition patterns in an HIV--immune system model (Wodarz--Nowak)}
\label{subsec:app-hiv-nowak-wodarz}

We illustrate the structural transition framework using the classical
Wodarz--Nowak HIV model, previously analyzed from a reaction--network
perspective in \cite{dittrich2007chemical}. The reaction network consists of
four species
\[
\mathcal S=\{x,y,w,z\},
\]
representing healthy target cells $x$, infected cells $y$, CTL precursors $w$,
and CTL effectors $z$.

The reaction rules are given by
\[
\begin{aligned}
\emptyset &\to x \\
x+y &\to 2y \\
x &\to \emptyset \\
y+z &\to z \\
y &\to \emptyset \\
x+y+w &\to x+y+2w \\
w &\to \emptyset \\
y+w &\to y+z \\
z &\to \emptyset .
\end{aligned}
\]

The parameters and initial conditions used in the simulations are listed in
Table~\ref{tab:hiv_params}.

\begin{table}[h]
\centering
\caption{Parameters and initial conditions for HIV model simulations}
\label{tab:hiv_params}
\begin{tabular}{llr}
\toprule
Parameter & Description & Value \\
\midrule
$\lambda$ & Production rate of healthy T cells & 10 \\
$d$ & Death rate of uninfected T cells & 0.01 \\
$\beta$ & Infection rate & 0.00001 \\
$a$ & Death rate of infected T cells & 0.5 \\
$p$ & CTL proliferation rate & 0.1 \\
$c$ & CTL killing rate & 0.01 \\
$b$ & CTL differentiation rate & 0.01 \\
$h$ & CTL precursor death rate & 0.01 \\
$q$ & CTL effector death rate & 0.1 \\
\midrule
$x(0)$ & Initial healthy cells & 1000 \\
$y(0)$ & Initial infected cells & 1 \\
$w(0)$ & Initial CTL precursors & 0 \\
$z(0)$ & Initial CTL effectors & 0 \\
\midrule
$t_{\text{inj}}$ & Injection time & 2000 \\
$w_{\text{weak}}$ & Weak injection amplitude & 0.1 \\
$w_{\text{strong}}$ & Strong injection amplitude & 10 \\
$\varepsilon$ & Presence threshold & $10^{-8}$ \\
\bottomrule
\end{tabular}
\end{table}

Starting from the initial presence set
\[
\phi_\varepsilon(c(0))=\{x,y\},
\]
the CTL precursor species $w$ is injected at time $t_{\mathrm{inj}}$
with two different amplitudes (weak versus strong). Presence sets are extracted
using the thresholded abstraction
\[
\phi_\varepsilon(c(t))
:=
\{\,s_i\in\mathcal S\mid c_i(t)>\varepsilon\,\},
\qquad
\varepsilon=10^{-8}.
\]

The resulting concentration dynamics and induced transition patterns are shown
in Figure~\ref{fig:hiv-graphs}.

\begin{figure}[h]
    \centering
 \includegraphics[trim=0 0 0 0, clip,width=1\textwidth]{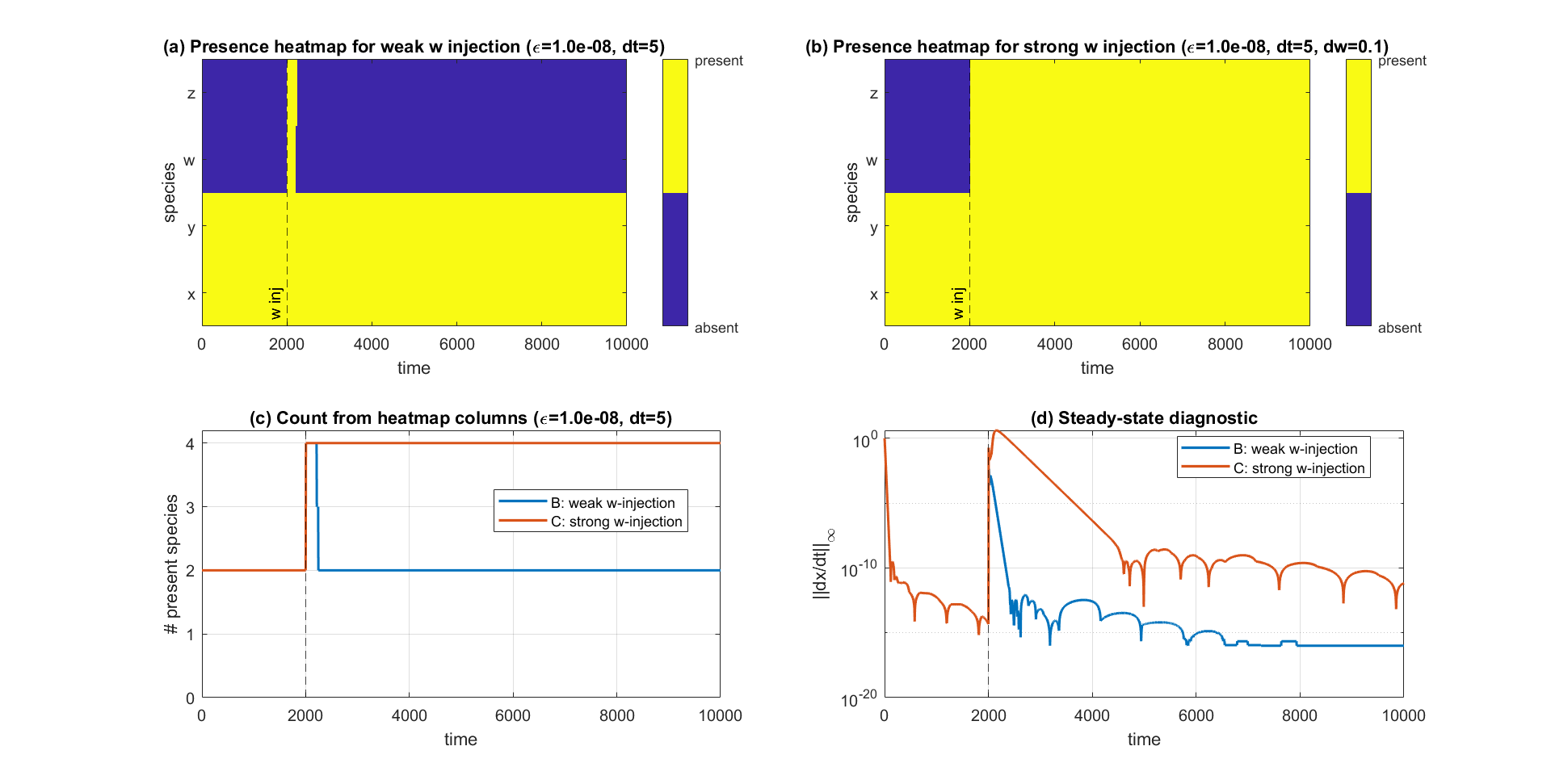}
   		\caption{Numerical simulations of the Wodarz--Nowak HIV model illustrating
ODE--induced set--valued transitions.
Shown are concentration heatmaps of the four species $x,y,w,z$ for two trajectories
starting from the same initial condition and the same initial presence set
$\phi_\varepsilon(c(0))=\{x,y\}$.
At time $t_{\mathrm{inj}}$, the CTL precursor species $w$ is injected with two different
amplitudes (weak vs.\ strong).
In both cases, the solution first induces the same closure-induced upward transition to the
closed set $\mathrm{clos}(\{x,y,w\})=\{x,y,w,z\}$.
For the strong injection, the trajectory converges to the terminal organization
$\{x,y,w,z\}$, whereas for the weak injection, a subsequent downward transition occurs
and the system returns to the organization $\{x,y\}$.
Presence sets are extracted using the thresholded abstraction
$\phi_\varepsilon(c(t))$ with $\varepsilon=10^{-8}$.}
		\label{fig:hiv-graphs}
\end{figure}

Both perturbations induce the same canonical upward transition
\[
\{x,y\}\to\mathrm{clos}(\{x,y,w\})=\{x,y,w,z\},
\]
in agreement with Proposition~\ref{cor:canonical-closure}. For the strong
injection, the trajectory converges to the terminal organization
\[
\{x,y,w,z\},
\]
whereas for the weak injection, species $z$ and later $w$ lose reaction support,
leading to the terminal organization
\[
\{x,y\}.
\]

The intermediate set $\{x,y,w\}$ is a non--organization and therefore transient.
The example illustrates the asymmetry established in
Theorem~\ref{thm:ode-induced-transitions}: closure-building transitions are structurally
canonical, whereas downward transitions and terminal organizations depend on the
realized trajectory.

The induced set--valued transition structure is summarized in
Figure~\ref{fig:hiv-tikz}.

\begin{figure}[t]
\centering
\scalebox{0.9}{
\begin{tikzpicture}[
  >=latex,
  node distance=10mm,
  setnode/.style={draw, rounded corners, align=center, font=\small, inner sep=4pt},
  up/.style={->, thick},
  down/.style={->, thick, dashed},
  lab/.style={font=\footnotesize, inner sep=1pt, fill=white}
]
\node[setnode] (S0) {organization $S_0=\{x,y\}$};
\node[setnode, above=of S0] (S1) {non-closed $S_1=\{x,y,w\}$};
\node[setnode, above=of S1] (S2) {organization $S_2=\{x,y,w,z\}=\mathrm{clos}(S_1)$};

\draw[up] (S0) -- node[lab, pos=0.5, anchor=east] {$T_{0\to1}$: inject $w$ at $t=2000$} (S1);
\draw[up] (S1) to[bend left=25] node[lab, pos=0.5, anchor=east] {$T_{1\to2}$: building closure} (S2);
\draw[down] (S2) to[bend left=25] node[lab, pos=0.5, anchor=west] {$T_{2\to1}$: depletion of $z$ (only trajectory B)} (S1);
\end{tikzpicture}
}
\caption{Set--valued transition structure induced by the HIV model simulations.
Starting from the initial presence set $\{x,y\}$, the system undergoes a canonical
upward transition to the closed set $\{x,y,w,z\}$.
Depending on the realized trajectory, the system either remains in the organization
$\{x,y,w,z\}$ or returns to $\{x,y\}$ through a downward transition associated
with loss of reaction support.}
\label{fig:hiv-tikz}
\end{figure}
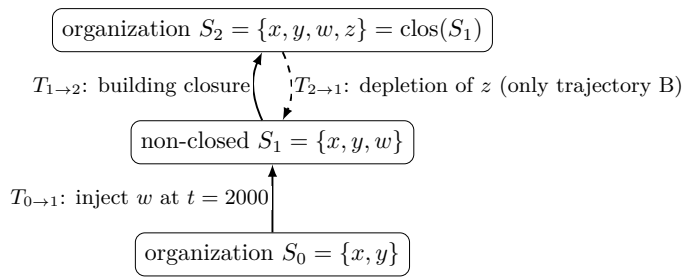

The same structural transition patterns can also be identified in other reaction--based
models, including virus infection systems and reaction--diffusion dynamics
\cite{Handel2009,peter2021linking,peter2024cell}.
The framework is therefore not restricted to a specific dynamical formalism but
provides a general structural description of admissible changes between species subsets.

\subsection{Structural transitions and Markov processes}
\label{subsec:markov}

The transition framework introduced above is purely structural and independent of
kinetic parameters or probabilistic assumptions. Nevertheless, it naturally constrains
stochastic descriptions of reaction networks.

Let $\mathcal{S}$ denote the set of species subsets and
$\mathcal{T}\subseteq\mathcal{S}\times\mathcal{S}$ the set of admissible structural
transitions. This transition graph specifies which changes of species presence are
structurally feasible under the reaction network.

Stochastic reaction network models are commonly formulated as continuous--time
Markov processes on discrete state spaces
\cite{Norris1998,Anderson2011}. In reduced descriptions, the states may correspond
to subsets of species \cite{Gillespie1977,Gillespie2007}. Markovian descriptions
have also been introduced at the level of chemical organizations to study their
evolution, complexity, and resilience
\cite{VelozGonzalez2023}.

The perspective adopted here is different. Rather than postulating a stochastic
process between organizations, we first determine a kinetics-independent graph of
structurally admissible transitions between arbitrary species subsets. A Markov
process on this state space can subsequently be obtained by assigning transition
probabilities to admissible transitions in $\mathcal{T}$.

Importantly, the probabilistic dynamics do not introduce new transitions:
every stochastic transition must correspond to a structurally admissible transition.
Different stochastic models therefore correspond to different probabilistic weightings
of the same underlying transition graph. In particular, upward transitions leading to
closures remain structurally determined, whereas downward transitions may occur with different
probabilities depending on noise or environmental fluctuations
\cite{Anderson2015}.

Thus, admissible transitions provide a structural constraint layer underlying both
deterministic and stochastic reaction network dynamics.


The results of this section establish a unified structural framework for transitions
between species subsets in reaction--based dynamical systems. Solutions of ODE
systems induce canonical sequences of feasibly admissible transitions with a
distinguished upward closure--building phase and a terminal phase leading to
organizations. The HIV example illustrates these mechanisms in a biological setting,
while the Markov case demonstrates that the same structural constraints also govern
stochastic dynamics. Together, these results show that admissible transitions capture
system--independent structural constraints on transient behavior at the level of
species presence.
\section{Discussion}

\subsection{Related work}
\label{sec:related-work}

Reaction network analysis has been developed along several complementary directions,
including chemical reaction network theory (CRNT), chemical organization theory (COT),
and deterministic or stochastic dynamical models. The present work is situated at
the interface of these approaches and focuses on transient structural changes between
species subsets.

Classical CRNT relates network structure to qualitative dynamical properties such as
stability and uniqueness of equilibria under mass--action kinetics
\cite{Feinberg1979lectures,Feinberg1987deficiency}. While CRNT provides powerful
tools for steady--state analysis, it does not describe transient changes of species
presence or admissible transitions between species subsets.

Recent work has emphasized the usefulness of chemical reaction network
concepts in mathematical epidemiology, including structural approaches to
epidemic models, persistence, boundary behavior, siphons, and autocatalytic
structures
\cite{Avram2024EpidemiologyCRN,Avram2026Siphons}.
These developments illustrate how reaction-network structure can reveal
qualitative properties of biologically motivated dynamical systems. The present
work complements this perspective by focusing specifically on structurally
admissible transient changes between species-presence sets.

Chemical organization theory characterizes structurally stable subsets of species
through the notions of closure and self--maintenance
\cite{dittrich2007chemical,peter2026chemical}. Organizations were later linked to
persistence properties of reaction--based dynamical systems
\cite{peter2021linking}. Earlier work within COT also considered the evolution of
chemical systems as movements through the space of organizations, distinguishing
upward, downward, and sideward organizational changes
\cite{Matsumaru2006}. These transitions, however, are formulated at the level of
chemical organizations. In contrast, the present framework defines admissible
transitions between arbitrary species subsets of a fixed reaction network and
characterizes their feasibility directly in terms of stoichiometry and reaction
support.

Other structural approaches include elementary flux mode analysis and Petri net
representations. Elementary modes characterize minimal steady--state flux patterns
\cite{schuster1994elementary}, while Petri nets provide tools for analyzing reachability
and invariants in biochemical systems
\cite{Feinberg1979lectures,reisig2013understanding}. These approaches focus either on
steady states or on concentration-based state spaces rather than on transitions
between species presence sets.

Reaction networks are also commonly studied using deterministic ODE systems and
stochastic Markov models. Such models describe trajectories and probabilistic
dynamics in detail. Moreover, previous work has considered organizational evolution,
perturbation and resilience, and stochastic transitions between chemical organizations
\cite{Matsumaru2006,VelozResilience2022,VelozGonzalez2023}.

The framework introduced here is complementary to these approaches. Its central
object is not a transition between already identified organizations, nor a
probabilistically weighted transition induced by a particular dynamical model.
Instead, admissibility is defined directly for arbitrary pairs of species subsets
and determined from reaction-network structure through stoichiometric sign constraints
and feasible reaction support. This provides a structural layer between network
topology and particular deterministic or stochastic realizations.

\subsection{Discussion}
\label{sec:discussion}

We introduced a structural framework for analyzing transient behavior in reaction
networks through admissible transitions between species subsets. A key aspect of the
framework is the separation between structural constraints imposed by network topology
and the particular dynamical realization induced by kinetic laws or stochastic effects.
Our results reveal a characteristic asymmetry: closure-building transitions are
structurally determined, whereas downward transitions leading to terminal
organizations are generally non-unique and path-dependent.

Chemical organization theory characterizes structurally stable species sets associated
with asymptotic persistence. In contrast, the present framework places organizations
within a broader transition landscape, where they appear as terminal states connected
through structurally constrained pathways. The HIV example in
Section~\ref{subsec:app-hiv-nowak-wodarz} illustrates that the same closed set
$\{x, y, w, z\}$ can lead either to viral control or immune collapse depending on the
realized admissible downward transition. This demonstrates that transient structural
states may influence long-term outcomes in biological systems.

The framework complements existing approaches such as chemical reaction network
theory, chemical organization theory, elementary flux mode analysis, and Petri net
approaches by focusing explicitly on admissible changes between species presence
sets rather than only on asymptotic states or concentration trajectories
\cite{peter2021linking}. In this way, admissible transitions provide a qualitative
description of transient structural behavior that is independent of detailed kinetic
parameters.

The framework also suggests several practical and computational perspectives.
Admissibility testing can be formulated as a feasibility problem for reaction fluxes,
while reachability analysis requires computing closures and reachable organizations.
Constructing complete transition graphs may become computationally demanding due
to the potentially large number of organizations
\cite{peter2023computing,ruth2024revealing}. Nevertheless, such methods may
enable automated exploration of structural pathways in large reaction networks, for example in ecology~\cite{PeterIbrahim2026Dryland} or neurology~\cite{peter2024intuitive}.

Several extensions remain possible. Future work may include spatial reaction--diffusion
systems, adaptive or time-dependent reaction networks, and quantitative refinements
that incorporate flux magnitudes or transition probabilities. Experimental validation
in synthetic or biochemical systems could further clarify the relevance of admissible
transition pathways for biological dynamics.

The present framework is intentionally qualitative and focuses on species presence
rather than concentration values or time scales. Consequently, it does not replace
detailed dynamical analysis, but instead provides complementary structural information
about transient reachability and admissible system evolution.

Taken together, the proposed framework extends structural reaction network analysis
from static organizations to transient structural dynamics and provides a
kinetics-independent perspective on admissible pathways in biological and biochemical systems.

\section{Conclusion}
Admissible transitions provide a structural description of how species-presence sets
may change in reaction networks independently of detailed kinetic laws. The framework
connects continuous reaction-based dynamics with discrete transition sequences and
reveals an intrinsic asymmetry between canonical closure-building transitions and
generally path-dependent downward transitions. The HIV--immune system example shows
how distinct admissible pathways can be associated with different terminal
organizations and long-term biological outcomes. Thus, the proposed framework extends
structural reaction-network analysis from persistent configurations to transient
structural dynamics and offers a basis for future computational studies of
reachability and organization formation in biological systems.


%

\bibliographystyle{ws-jbs}
\bibliography{sample}

\end{document}